\documentclass[11pt,twoside,a4paper]{amsart}
\usepackage{amssymb}
\date{\today}

\def\vp{\varphi}

\def\ddzi{\frac{\partial}{\partial \z _{i}}}
\def\z{\zeta}

\def\reg{{\rm reg\,}}

\def\deg{\text{deg}\,}

\def\w{\wedge}

\def\dbar{\bar\partial}

\def\C{{\mathbb C}}
\def\w{{\wedge}}
\def\P{{\mathbb P}}
\def\Pk{{\mathbb P}}

\def\Ak{{\mathcal A}}
\def\Bk{{\mathcal B}}

\def\S{{\mathcal S}}

\def\Cu{{\mathcal C}}

\def\Hom{{\rm Hom\, }}
\def\codim{{\rm codim\,}}

\def\Im{{\rm Im\, }}
\def\Lop{{\mathcal L}}

\def\Kers{{\mathcal Ker\,  }}

\def\E{{\mathcal E}}

\def\Ok{{\mathcal O}}

\def\Q{{\mathcal Q}}
\def\Re{{\rm Re\,  }}

\def\J{{\mathcal J}}
\def\nbh{neighborhood }

\def\be{\begin{equation}}
\def\ee{\end{equation}}

\def\sko{{\vartheta}}

\newtheorem{thm}{Theorem}[section]
\newtheorem{lma}[thm]{Lemma}

\newtheorem{prop}[thm]{Proposition}

\theoremstyle{definition}

\newtheorem{df}{Definition}

\theoremstyle{remark}

\newtheorem{remark}[thm]{Remark}
\newtheorem{ex}[thm]{Example}

\numberwithin{equation}{section}

\title[]{Division formulas on  projective
varieties}

\begin{document}

\date{\today}

\author{Mats Andersson \& Lisa Nilsson}

\address{Department of Mathematics\\Chalmers University of Technology and the University of 
Gothenburg\\S-412 96 G\"OTEBORG\\SWEDEN}

\email{matsa@chalmers.se \& Lisa.nilsson@if.se}


\thanks{The first author was
  partially supported by the Swedish 
  Research Council. The second author was supported by a postdoc grant
from the Swedish Research Council.}

\begin{abstract}
We introduce a  division formula on a possibly singular 
projective subvariety 
$X$ of complex projective space $\Pk^N$,
which, e.g.,  provides explicit representations of solutions
to various polynomial division problems
on the affine part of $X$.
Especially  we consider a global effective version of
the Brian\c con-Skoda-Huneke theorem.
\end{abstract}

\maketitle

\section{Introduction}\label{intro}

In this paper we construct a  division-interpolation integral formula for polynomial
ideals on an algebraic, not necessarily smooth, subvariety $V$ of $\C^N$ of pure dimension $n$. 
Such division formulas,  for smooth $V$,
have  been used by several authors, see, e.g., \cite{BB},
\cite{BY} and \cite{BGVY}. In \cite{AG} were introduced division formulas on $\C^n$ 
where the integration ``takes place'' over the entire compactification
$\Pk^n$.  This made it possible to make  
a more careful analysis at infinity which led to explicit representation of
membership with sharper degree estimates than by earlier known formulas.  

Our new formula involves integration  over the
closure $X$ of $V$ in $\Pk^N$. 
It is quite general and might be used to represent solutions to
a variety of membership problems, but our focus is on 
a global effective Brian\c con-Skoda-Huneke type polynomial division result
on $V$,  \cite[Theorem A]{AWsemester}, generalizing 
theorems  of Ein-Lazarsfeld, \cite{EL},  and
Hickel, \cite{Hick}, to non-smooth $V$.  
Our formula 
does not in any substantial way contribute to the proof of this result so the novelty 
is that the membership is realized by an explicit formula. One could hope that this formula
keeps arithmetic information\footnote{as does the ingenious formula in \cite{BY} in case with
Nullstellensatz data and $X=\Pk^n$.}, for instance, that the coefficients are rational if the data have
rational coefficients, but we have not been able to prove this 
except for very special examples. 
One could also hope for estimates of the size of the coefficients. In any case
we think that the very explicitness is of interest in itself. 
If applied to Nullstellensatz data we get a representation
of a solution with a polynomial degree that  is not too far 
from the optimal one,  cf., Remark \ref{julle}  below.

Let $X$ be the closure of $V$ in $\Pk^N$, let $J_X$ 
denote the associated homogeneous ideal in the graded ring $S=\C[z_0,\ldots,z_N]$, 
and let $S(\ell)$ denote the module $S$ but where
all degrees are shifted by $\ell$.  Let
\begin{equation}\label{skottavla}
0\to S_M\stackrel{a_M}{\to} \cdots \stackrel{a_1}{\to} S_0
\end{equation}
be a minimal graded free resolution  of $S/J_X$;
here  $S_k=S(-d_k^1)\oplus\cdots\oplus S(-d_k^{r_k})$
and the  mappings $a_k=(a_k^{ij})$ are 
matrices of homogeneous forms  in $\C^{N+1}$ with
$$
\deg a_k^{ij}= d_k^j-d_{k-1}^i.
$$ 
Since $J_X$ has pure dimension it follows from 
\cite[Corollary~20.14]{Eis}  that $M\le N$, 
see also \cite[Section~2.7]{AWsemester}.
Recall  that the {\it regularity} of $X$ is by definition the regularity of the ideal
$J_X$, so that 
\begin{equation}\label{badkar}
\reg X:=\max_{k,i} (d_k^i-k) +1,
\end{equation}
see, e.g., \cite[Ch.~4]{Eis2}.

Given polynomials  $F_1,\ldots, F_m$ on $V$ of degrees at most $d$, 
let $f_j$ be the corresponding
$d$-homogeneous forms on $\C^{N+1}$. That is, if $z=(z_0,\ldots,z_N)$ and $z'=(z_1,\ldots,z_N)$,
then $f_j(z)=z_0^d F(z'/z_0)$. As usual we can consider $f_j$ as sections of the restriction
to $X$ of the line bundle $\Ok(d)$ over $\Pk^N$. Let $Z_f$ be the common zero set of  the
$f_j$ on $X$. 

We now describe our main result in this paper.
Assume that  $\rho$ is an integer  that is larger than or equal to 
$d\min(m,n+1)+\reg X-1$,
and let $\phi$ be  a section of $\Ok(\rho)$ over $X$.  Let us assume in addition that $X$ is irreducible, cf., Remark~\ref{mora} below.
We then have an explicit division-interpolation  formula
\begin{equation}\label{divint}
\phi(z)=\sum_{j=1}^m f_j(z)\int_X \Ak^\rho_j(\zeta,z)\phi(\zeta)+
\int_X \Bk^\rho(\zeta,z)\w  R^f(\zeta)\w\omega(\zeta)\phi(\zeta),
\end{equation}
where,  for fixed $z$, the  integrals in the sum exist as  principal values
at $X_{sing}\cup Z_f$, $\Bk^\rho(\cdot,z)$ is a smooth form,  and $R^f\w\omega$ is a residue current on $X$ 
with support on $Z_f$, so that the second ``integral'' is the action of $R^f\w\omega$ 
on $\pm \Bk^\rho \phi$.
Moreover, both $\Ak^\rho_j$ and $\Bk^\rho$ are holomorphic (homogeneous polynomials
of degrees $\rho-d$ and $\rho$, respectively) in $z$.
For the precise formula, see Theorem~\ref{puma} below.

Assume that  $\Phi$ is a polynomial
of degree at most $\rho$ and let $\phi$ be its  $\rho$-homogenization.
If the current $R^f\w\omega\phi$ vanishes, i.e., $\phi$ annihilates $R^f\w\omega$,
it follows that
\begin{equation}\label{qprim}
Q_j(z')=\int_X \Ak^\rho_j(\zeta; 1,z')\phi(\zeta)
\end{equation}
are polynomials such that
$\deg F_jQ_j\le \rho$ and
$F_1Q_1+\cdots +F_mQ_m= \Phi$
on $V$.

\smallskip
Let us now present our main application of this formula.  
Let  $\J_f$ be the coherent analytic ideal sheaf over $X$ generated by $f_j$.
Furthermore, 
let $c_\infty$ be the maximal codimension of  the
so-called  {\it distinguished varieties} of the sheaf $\J_f$,
in the sense of Fulton-MacPherson,   
that are contained in 
$$
X_\infty:=X\setminus V,
$$
see, e.g., \cite[Section~5]{AWsemester}.
If there  are no distinguished varieties contained in  $X_\infty$,
then we interpret $c_\infty$ as $-\infty$.
We always have that 
$$
c_\infty\le\mu:=\min(m,n).
$$
Let $|F|^2=|F_1|^2+\cdots +|F_M|^2$.

\begin{thm}\label{hickelsats}
Assume that  $V$ is a reduced  $n$-dimensional algebraic subvariety  of $\C^N$, let $X$ be its closure in $\Pk^N$, and assume that $X$ is irreducible.

\smallskip
\noindent  (i)    There exists a number  $\mu_0$
with the following property: If $F_1,\ldots,F_m$ are polynomials of degree
$\le d$, $\Phi$ is a polynomial such that
\begin{equation}\label{kaka}
|\Phi|/|F|^{\mu+\mu_0} \ {\rm is\ locally\  bounded\  on} \ V,
\end{equation}
and
\begin{equation}\label{rodef}
\rho=\max \big(\deg \Phi + (\mu+\mu_0) d^{c_\infty} \deg X, d\min(m,n+1) +\reg X-1\big),
\end{equation}
then 
Eq.\ \eqref{qprim} defines explicit polynomials $Q_j$ such that
\begin{equation}\label{hummer}
\Phi=F_1Q_1+\cdots +F_mQ_m
\end{equation}
on $V$ and 
\begin{equation}\label{polupp}
\deg (F_j Q_j)\leq \rho 
\end{equation}

\smallskip
\noindent (ii) If in addition $V$ is smooth there is a number $\mu'$ with the following property: 
If $F_1,\ldots,F_m$ are polynomials of degree
$\le d$, $\Phi$ is a polynomial such that
\begin{equation}\label{kaka2}
|\Phi|/|F|^{\mu} \ {\rm is\ locally\  bounded\  on} \ V,
\end{equation}
and
\begin{equation}\label{rodef2}
\rho'=\max \big(\deg \Phi + \mu  d^{c_\infty}\deg X + \mu', d\min(m,n+1)+\reg X-1),
\end{equation}
then Eq.\ \eqref{qprim} defines explicit polynomials $Q_j$ 
such that  \eqref{hummer} holds on $V$ and 
\begin{equation}\label{polupp2}
\deg(F_j Q_j)\leq  \rho' 
\end{equation}
If $X$ is smooth one can take $\mu'=0$.
\end{thm}

\begin{remark} 
In \cite[Theorem A]{AWsemester} the degree estimates are slightly sharper; 
instead of $d\min(m,n+1)+\reg X-1$
in \eqref{polupp} and \eqref{polupp2} one has 
the number  $(d-1)\min(m,n+1)+\reg X$. 
\end{remark}

\begin{remark}\label{mora} 
Notice that if $V$ is irreducible in $\C^N$, then $X$ is irreducible.
In \cite[Theorem A]{AWsemester} there is no assumption about irreducibility of 
$X$.  Formula \eqref{divint}, as well as Theorem \ref{hickelsats}, 
hold without this assumption if we replace the bound $d\min(m,n+1)+\reg X-1$ by
$d\min(m,N+1)+\reg X-1$. The construction is precisely the same; we only use
the irreducibility assumption in Lemma \ref{skatan}, which gives 
rise to the slightly sharper bound.
\end{remark}

\begin{remark}\label{julle}
If we apply  Theorem~\ref{hickelsats} to  Nullstellensatz  data,
i.e., $F_j$ with no common zero on $V$
and $\Phi=1$, then we get polynomials $Q_j$ by  \eqref{qprim} such that  
 $\sum F_jQ_j=1$ on $V$
and 
$$
\deg(F_jQ_j)\le
\max \big((\mu+\mu_0) d^{c_\infty} \deg X, d\min(m,n+1) +\reg X-1 \big).
$$
It was proved by Jelonek, \cite{Jel}, that one can find a solution such that 
$$
\deg(F_jQ_j)\le c_m d^\mu\deg X,
$$
where $c_m=1$ if $m\le n$ and $c_m=2$ otherwise, and this result
is essentially optimal.
In general it  is clearly much sharper than what we get; however,
if  $c_\infty<\mu$ and $d$ is large enough, then our estimate may be 
sharper\footnote{The optimal result for case $V=\C^n$ was proved by Koll\'ar,  \cite{Koll}, for
$d\ge 3$; see  \cite{Sombra} and \cite{Jel} for $d=2$.}. 
\end{remark}

\smallskip
The main point in this paper is the construction of the formula \eqref{divint}
and it is performed step by step in Sections~\ref{formelsec}, \ref{gam}, and \ref{basic}.
In Section~\ref{hypers} we consider the special case when $X$ is a hypersurface
in $\Pk^{n+1}$. When reading the general construction it might be instructive to 
look into this section to see what each step boils down to in the hypersurface case. 
As soon as  formula \eqref{divint} is established, Theorem~\ref{hickelsats}
is a simple consequence of the main step  in proof of Theorem~A in \cite{AWsemester};
this is explained in  Section~\ref{ponton}.

\smallskip
\noindent
{\bf Acknowledgement}  We are grateful for the referee's careful reading and suggestions
to improve the presentation.

\section{Integral  representation on $\Pk^N$}\label{formelsec}

Let 
$$
\pi\colon\C^{N+1}\setminus\{0\}\to\Pk^N
$$
be the natural projection. Let $U$ be an open subset of $\Pk^N$. Recall that differential forms 
(currents) with values in
$\Ok(\ell)$ can be identified by 
$\ell$-homogeneous differential forms (currents)  $\xi$ on $\pi^{-1}U\subset \C^{N+1}\setminus\{0\}$ 
such that $\delta_\zeta \xi=\delta_{\bar \zeta}\xi=0$, where
$\delta_\zeta$ is interior multiplication by
$$
\sum_0^N \zeta_j\frac{\partial}{\partial \zeta_j},
$$
and $\delta_{\bar \zeta}$ is its conjugate.  
We will use  the norm $|\xi|_{\Ok(\ell)}=|\xi|/|\zeta|^\ell$.

We first recall from \cite{AG} how one can generate
representation formulas for  holomorphic sections of a  vector bundle  $F\to\Pk^N$.
The construction is an adaptation to $\Pk^N$ of the idea introduced  in 
 \cite{A1}; see also \cite{EG}.
Let $F_z$ denote the pull-back of $F$  to $\Pk^N_\zeta\times \Pk^N_z$ under the
natural projection $\Pk^N_\zeta\times \Pk^N_z\to \Pk^N_z$  
and define  $F_\zeta$  analogously.
Let $\delta_\eta$ denote contraction with the  
$\mathcal{O}_z(1) \otimes \mathcal{O}_\zeta (-1)$-valued holomorphic vector field 
$$
\eta = 2 \pi i \sum_0^N z_i \ddzi
$$
over $\Pk^N_\zeta\times\Pk^N_z$. We thus have the mapping
$$
\delta_\eta : \Cu_{\ell+1,q}(\mathcal{O}_\z(k) \otimes \mathcal{O}_z(j)) \to
  \Cu_{\ell,q}(\mathcal{O}_\z(k-1) \otimes \mathcal{O}_z(j+1)),
$$
where $\Cu_{\ell,q}(\mathcal{O}_\z(k) \otimes \mathcal{O}_z(j))$ denotes the
sheaf of currents on $\Pk^N_\zeta\times\Pk^N_z$
of bidegree $(\ell,q)$ in $\zeta$ and $(0,0)$ in $z$
that take  values in $\mathcal{O}_\z(k) 
\otimes \mathcal{O}_z(j)$. 
In this paper we only deal with  forms and currents with respect to
$\zeta$ and always consider  $z$ as a parameter, i.e., no differentials
with respect to $z$ occur.

Given a vector bundle $L\to\Pk^N_\zeta\times\Pk^N_z$, let
\[
\Lop^{\nu}(L)=
\bigoplus_{j}\Cu_{j,j+\nu}(\mathcal{O}_\z(j) \otimes \mathcal{O}_z(-j)\otimes  L).
\]
If $\nabla_\eta = \delta_\eta - \dbar$, 
where $\dbar=\dbar_\zeta$, 
then  $\nabla_\eta: \Lop^{\nu}(L) \to \Lop^{\nu+1}(L)$. Furthermore,
$\nabla_\eta$ is a anti-derivation,  and $\nabla_\eta^2 = 0$.

A  {\it weight} with respect to $F\to \Pk^N$  and a point\footnote{Usually
"$z$ in $\Pk^N$" means that $z\in\C^{N+1}\setminus\{0\}$ and the point in question
is $\pi(z)$ in  $\Pk^N$''.}
$z$ in  $\Pk^N$  is a section  $g$ of $\Lop^{0}(\Hom(F_\zeta,F_z))$ such that $\nabla_\eta g =
  0$, $g$ is smooth for $\zeta$ close to $z$, 
and  $g_{0,0} = I_F$ when $\zeta=z$, 
where $g_{0,0}$ denotes the component of  $g$ with bidegree $(0,0)$,
and $I_F$ is the identity endomorphism on $F$. 
The following basic formula appeared as Proposition~4.1 in 
\cite{AG}.

\begin{prop} \label{hatsuyuki}
Let $g$ be a weight with respect to $F\to\Pk^N$ and  $z$  and assume that
$\psi$ is  a holomorphic section of $F\otimes\mathcal{O}(-N)$.
We  then have  the representation formula
\begin{equation}\label{rep1}
\psi(z) = \int_{\Pk^N} g_{N,N} \psi. 
\end{equation}
\end{prop}

Fix $z$ and assume that  $g$ and $g'$ are weights with respect to $F$ and $\Ok(\ell)$ respectively. 
Then $g\w g'$ is a weight with respect to $F\otimes\Ok(\ell)$. In fact,
$\nabla_\eta(g\w g')=\nabla_\eta g\w g'+ g\w\nabla_\eta g'=0$, and
$(g\w g')_{0,0}=g_{0,0}g'_{0,0}=I_F \cdot 1=I_{F\otimes\Ok(\ell)}$.

\begin{ex}\label{alphaex}
Fix a point $z$. Notice that 
\[
\alpha = \alpha_{0,0} + \alpha_{1,1} := \frac{z \cdot \bar \z}{|\z|^2} - \dbar
\frac{\bar \z \cdot d\z}{2 \pi i |\z|^2}
\]
is a well-defined smooth form in $\Lop^0(\Hom(\Ok_\zeta(-1),\Ok_z(1)))$
such that 
\begin{equation}\label{alfasluten}
\nabla_\eta\alpha=0
\end{equation}
and $\alpha_{0,0}$ is equal to $1=I_{\mathcal{O}(1)}$ at $\zeta=z$.
Thus $\alpha$ is a weight with respect to $\Ok(1)$ and $z$.  If $g$ is a weight with respect
to  $F$, thus $g\w\alpha^\ell$ is a weight with respect to $F\otimes\Ok(\ell)$.
\end{ex}

For a given point $z\in\Pk^N$, 
$$
b=\frac{1}{2\pi i}\frac{|\zeta|^2\bar z\cdot d\zeta-(\bar z\cdot\zeta)\bar\zeta\cdot d\zeta}
{|\zeta|^2|z|^2-|\bar\zeta\cdot z|^2}
$$
is the $\Ok_\zeta(1)\otimes \Ok_z(-1)$-valued $(1,0)$-form
on $X\setminus\{z\}$
of minimal norm such that $\eta \cdot b=\delta_\eta b=1$. Let 
$$
B=\frac{b}{\nabla_\eta b}=b+ b\w\dbar b+ b\w(\dbar b)^2+\cdots +b
\w(\dbar b)^{N-1};
$$
here  $b\w(\dbar b)^{k-1}$ is an 
$\Ok_\zeta(k)\otimes \Ok_z(-k)$-valued $(k,k-1)$-form that is
$\Ok(1/d(\zeta,z)^{2k-1})$, where $d(\zeta,z)$ is the distance between
$\zeta$ and $z$ on $\Pk^N$.  In particular,  $B$ is locally
integrable at $z$ and can thus  be considered as a current on $X$.
Let $[z]$ denote the $\Ok_\zeta(N)\otimes \Ok_z(-N)$-valued
$(N,N)$-current point evaluation at $z$, i.e.,
$$
\int_{\Pk^N} [z] \xi=\xi(z)
$$
for each smooth $(0,0)$-from $\xi$ with values  in $\Ok(-N)$.
Then, see, e.g., \cite[formula (4.1)]{AG},
\begin{equation}\label{pelle}
\nabla_\eta B=1-[z]
\end{equation}
in the current sense.  

For future reference we recall from
\cite{AG} how \eqref{rep1}
follows from \eqref{pelle}:  Since $g$ is $\nabla_\eta$-closed,  
$\nabla_\eta(g\w B)=g\w(1-[z])$.  Identifying components of bidegree
$(N,N)$ we therefore get
$$
-\dbar(g\w B)_{N,N-1}=g_{N,N}-g_{0,0}[z]=g_{N,N}-I_{F_z}[z].
$$
Since $\psi$ is holomorphic, we thus have
$$
-d(g\w B\psi)_{N,N-1}=-\dbar(g\w B\psi)_{N,N-1}=g_{N,N}\psi-[z]\psi.
$$
Integrating over $\Pk^N$ now gives \eqref{rep1}.

\subsection{Division-interpolation formulas on $\Pk^N$}\label{krut}

Assume that 
\begin{equation}\label{ecomplex}
0\to E_M\stackrel{\vp_M}{\longrightarrow}\ldots\stackrel{\vp_3}{\longrightarrow} 
E_2\stackrel{\vp_2}{\longrightarrow}
E_1\stackrel{\vp_1}{\longrightarrow}E_0\to 0
\end{equation}
is a generically exact complex of Hermitian holomorphic vector bundles  over $\Pk^N$ and
let $Z$ be the projective variety  where \eqref{ecomplex}  is not pointwise exact. 
We introduce  a superbundle
structure on $E=\oplus E_k$ so that elements in $E^+=\oplus E_{2k}$ are even and elements
in  $E^-=\oplus E_{2k+1}$ are odd. Then $\vp=\oplus_k \vp_k\colon  E\to E$
is mapping of odd order since it maps $E^{\pm}\to E^{\mp}$.  We get an induced superbundle 
structure on the space of forms (and currents) with values in $E$ by just adding
the degrees (mod $2$). If $v$ is an odd section of $E$ and $\xi$ is a form, then
$v\xi=(-1)^{\deg \xi} \xi v$  etc.  Thus for instance
$\vp (\xi v)=(-1)^{\deg \xi} \vp v$.
It follows also that $\dbar$ is an odd mapping, and
since $\vp$ is holomorphic we have that $\dbar\circ \vp=-\vp\circ\dbar$.
See \cite{AG} and \cite[Section~5]{A7} for details. 

In \cite{A7} and \cite{AW1} were  introduced   currents 
$$
U=U_1+\ldots +U_N+U_{\min(M,N+1)}, \quad R=R_1+\ldots+ R_{\min(M,N+1)}
$$
associated to \eqref{ecomplex} with the following properties:
The current $U$ is smooth outside $Z$, $U_k$ are $(0,k-1)$-currents
that take values in $\Hom(E_0, E_k)$,  and
$R_k$ are  $(0,k)$-currents with support on $Z$, taking values in 
$\Hom(E_0, E_k)$. They satisfy the relations
$$
\vp_1 U_1=I_{E_0}, \quad  \vp_{k+1}U_{k+1}-\dbar U_k= -R_k,\ k\ge 1,
$$
which can be compactly written 
\begin{equation}\label{alban}
\nabla_\vp U=I_{E_0}-R
\end{equation}
if
$
\nabla_\vp=\vp-\dbar=\vp_1+\vp_2+\cdots \vp_N-\dbar.
$
Moreover, $R_k=0$ for $k<\codim Z$.
Let  
\begin{equation}\label{utsikt}
0\to \mathcal{O}(E_M)\stackrel{\vp_N}{\longrightarrow}\ldots\stackrel{\vp_3}{\longrightarrow} 
\mathcal{O}(E_2)\stackrel{\vp_2}{\longrightarrow}
\mathcal{O}(E_1)\stackrel{\vp_1}{\longrightarrow}\mathcal{O}(E_0)
\end{equation}
be the corresponding complex of locally free sheaves.
In this paper we will only consider $E_k$ that are direct sums of line bundles
$\Ok(\nu)$.
Let $E_{k}^j$ be disjoint trivial line bundles over $\Pk^N$ 
with basis elements $e_{k,j}$, and let
\begin{equation}\label{dirsum}
E_k=\big(E^1_k\otimes \mathcal{O}(-d^1_k)\big)\oplus\cdots \oplus 
\big(E^{r_k}_k\otimes \mathcal{O}(-d_k^{r_k})\big).
\end{equation}
Then 
$$
\vp_k=\sum_{ij} \vp_k^{ij}e_{k-1,i}\otimes e_{k,j}^*
$$
are matrices of homogeneous forms;  here $e_{k,j}^*$ are the dual basis elements,  and 
$$
\deg \vp_k^{ij}= d_{k-1}^i-d_k^j.
$$
We equip $E_k$ with the natural  Hermitian metric, i.e., such that 
\begin{equation}\label{norm}
|\xi(z)|^2_{E_k}=\sum_{j=1}^{r_k}|\xi_j(z)|^2 |z|^{2d^j_k},
\end{equation}
if  $\xi=(\xi_1,\ldots,\xi_{r_k})$.

\smallskip
If $\psi$ is  a holomorphic section of $\mathcal{O}(E_0)$ 
that annihilates $R$, i.e., the current $R\psi$ vanishes,  then $\psi$ is 
in the sheaf $\J=\Im \vp_1$, see \cite[Proposition~2.3]{AW1}.    
In order to represent the membership by an integral formula 
we will use a weight $g$ that contains $\vp_1(z)$ as a factor
and apply Proposition \ref{hatsuyuki}. 
To form such a weight we need  
a generalization to nontrivial  vector bundles, introduced in \cite{AG}
and inspired by 
 \cite{A7} and \cite{AW1},  of so-called Hefer forms.

\df{We say that $H= (H_k^\ell)$ is a Hefer morphism for the complex
  $E_\bullet$  in \eqref{ecomplex}  if $H_k^\ell$ are smooth sections of
\[
\Lop^{-k+\ell}(\Hom(E_{\zeta,k},E_{z,\ell}))
\]
that are holomorphic in $z$, 
$H_k^\ell=0$ for $k<\ell$, the term $(H_\ell^\ell)_{0,0}$ of bidegree $(0,0)$
is the identity $I_{E_\ell}$ on the diagonal $\Delta$, and
\begin{equation}\label{hrel}
\nabla_\eta H_k^\ell =H_{k-1}^\ell \vp_k -\vp_{\ell+1}(z) H_k^{\ell+1}, 
\end{equation}
where $\vp_k$ stands for  $\vp_k(\zeta)$.}

\smallskip

\begin{remark}
Notice that 
$H$ is even. In fact the term $H^\ell_k$  is a form of degree
$k-\ell$ (mod~$2$)  that takes values in $\Hom(E_\ell,E_k)$,
so that its total degree is $(k-\ell) +(k-\ell)=0$ (mod~$2$).
\end{remark}

We do not require $H$ to be holomorphic in $\zeta$.
Assume that $H$ is a Hefer morphism for $E_\bullet$ and let 
$U$ and $R$ be the associated currents. We can then form the
currents  $HU=\sum_j H^1_jU_j$ 
(notice the superscript $1$ here)  and $HR=\sum_j H^0_j R_j$.

\begin{prop}\label{burdus}
Assume that  $H$ is a Hefer morphism for the complex $E_\bullet$
and assume that $E_0$ is a line bundle.
If $\psi$ is a holomorphic section of $F\otimes E_0 \otimes \mathcal{O}(-N)$
and $g$ is a weight with respect to $F\to\Pk^N$,
then  we have  the representation
\begin{equation}\label{urtva}
\psi(z)=\vp_1(z)\int_{\Pk^N_\zeta}(HU\w g)_{N,N}\psi+
\int_{\Pk^N_\zeta}(HR\w g)_{N,N}\psi, \quad z\in\Pk^N.
\end{equation}
\end{prop}

This proposition is the same as Proposition~4.2 in \cite{AG}
except for the additional weight $g$ which is easily dealt with. 
However, for future reference we will discuss a
proof, slightly different from the proof given in \cite{AG}.

\begin{proof}
Let $h$ be a non-trivial holomorphic section, of some Hermitian vector bundle, that vanishes on $Z$. 
If $\Re\lambda \gg 0$,  then
$$
U^\lambda:=|h|^{2\lambda} U,    \quad R^{\lambda}:=1-|h|^{2\lambda} + \dbar|h|^{2\lambda}
\w U
$$
are  (quite)  smooth forms.  Moreover  they admit  current-valued
analytic continuations to $\Re\lambda>-\epsilon$, and
the values at $\lambda=0$ are $U$ and $R$, respectively, see, e.g., \cite[Section~2]{AW1}. 
It is readily checked that 
\begin{equation}\label{skrut}
\nabla_\vp U^\lambda=I_{E_0}-R^\lambda.
\end{equation}

\begin{lma}\label{pluto}
The form 
\begin{equation}\label{regul}
g^\lambda:=\vp_1(z)HU^\lambda+HR^\lambda
\end{equation}
is a weight with respect to $E_0$ as long as $\Re\lambda\gg 0$.  
\end{lma}

\begin{proof} We must verify that $\nabla_\eta g^\lambda=0$ and that $g_{0,0}^\lambda=I_{E_0}$.
The analogous statements for $\lambda=0$ is part of Proposition~4.2 in \cite{AG}. The
verification there only relies on the equality \eqref{skrut}  for $\lambda=0$, 
i.e., \eqref{alban}, and
so  the same argument proves Lemma~\ref{pluto}.
\end{proof}

If now $g$ is a weight with respect to $F$, then  $g^\lambda\w g$
is a weigh with respect to $E_0\otimes F$. In view of Proposition~\ref{hatsuyuki}
we thus have the representation
$$
\psi(z)=\int_{\Pk^N} (g^\lambda\w g)_{N,N} \psi=
\vp_1(z)\int_{\Pk^N_\zeta}(HU^\lambda\w g)_{N,N}\psi+
\int_{\Pk^N_\zeta}(HR^\lambda\w g)_{N,N}\psi
$$
for a holomorphic section $\psi$ of $F\otimes E_0\otimes\Ok(-N)$ and $\Re\lambda\gg 0$.
However, both terms on the right hand side admit analytic continuations to 
$\Re\lambda>-\epsilon$, and taking $\lambda=0$ we get Proposition~\ref{burdus}.  
\end{proof}


Now assume that $g$ depends holomorphically on $z$.
If $R\psi=0$ we then get from Proposition \ref{burdus} the  explicit holomorphic solution
$$
q(z)=\int_{\Pk^N_\zeta}(HU\w g)_{N,N}\psi
$$
to $\vp_1 q=\psi$.  Moreover, if $\psi$ a~priori is only defined in a \nbh of $Z$, 
then the second integral   in
\eqref{urtva} is well-defined and provides a global holomorphic section $\tilde\psi$ such that
$\tilde\psi-\psi$ belongs to the ideal sheaf generated by $\vp_1$.  This is the reason for
the notion  division-interpolation formula.  For a more precise interpolation result,
see Proposition \ref{primat} below.

\begin{remark}\label{mainformel2}
In \cite{AW1} more general currents
$U^\ell_k$ and $R^\ell_k$,  taking values in $\Hom(E_\ell,E_k)$,
occur.  
With the same argument as above we get the more general formula
\begin{multline*}
\psi(z)=\vp_{\ell+1}(z)\int_{\Pk^N_\zeta}(H U^{\ell}\w g)_{N,N}\psi+
\int_{\Pk^N_\zeta}(H R^\ell\w g)_{N,N}\psi+\\
\int_{\Pk^N_\zeta}(H U^{\ell-1}\w g)_{N,N}  \vp_\ell\psi
\end{multline*}
for holomorphic sections of $F\otimes E_\ell\otimes\mathcal{O}(-N+\ell)$;
here $HR^\ell=\sum_jH^\ell_j R^\ell_j$ and $HU^\ell=\sum_j H^{\ell+1}_j U^\ell_j$.
If  $\vp_\ell\psi=0$ and $R^\ell\psi=0$ 
we thus get  an explicit  holomorphic solution to $\vp_{\ell+1}q=\psi$. 
\end{remark}

\subsection{A choice of Hefer morphism}\label{hefersubsec}

Assume now that $E_\bullet$ is a complex with $E_k$ of the form \eqref{dirsum}
and choose $\kappa$ such that $\kappa\ge d^i_k$ for all $i,k$. We shall 
construct a Hefer morphism  for the complex $E_{\bullet}\otimes\Ok(\kappa)$.

From the complex $E_\bullet$ we form a complex $E'_\bullet$ of trivial bundles over $\C^{N+1}$
in the following way: Let  $E'_k:=E_k^1\oplus\cdots\oplus E_k^{r_k}$ and take
the mappings $\vp'_k$ that are formally just the matrices $\vp_k$ of homogeneous forms.
Let  $\delta_{w-z}$ denote interior multiplication with
$$
2\pi i\sum_0^N (w_j-z_j)\frac{\partial}{\partial w_j}
$$
in $\C^{N+1}_w\times\C^{N+1}_z$.

\begin{prop} \label{kowalski}
There exist $(k-\ell,0)$-form-valued 
mappings 
$$
h_k^\ell=\sum_{ij}(h_k^\ell)_{ij} e_{\ell i} \otimes e_{k j}^\ast :
\C^{n+1}_w\times\C^{n+1}_z\to\Hom(E'_k , E'_\ell),
$$
such that 
\be \label{fraser}
h_k^\ell = 0,  \ \text{for}\ k<\ell,  \ h_\ell^\ell = I_{E'_\ell}, \ 
-\delta_{w-z} h_k^\ell =h_{k-1}^\ell \vp'_k(w) - \vp'_{\ell+1}(z) h_k^{\ell+1},
\ee
and the coefficients in the form 
$(h_k^\ell)_{ij}$ are homogeneous polynomials of degree
$d_{k}^j-d_{\ell}^i-(k-\ell)$. 
\end{prop}

In \cite[Section~4]{A7} there is an explicit
formula that provides  $h^\ell_k$ such that
\eqref{fraser} holds\footnote{The initial minus sign here is because
we have $\delta_{w-z}$ rather than $\delta_{z-w}$ as in \cite{A7}}.  
The  components of these forms of the desired
degrees then must satisfy \eqref{fraser} as well.  One can
check that the formula actually gives the desired forms directly. 
Notice that  
\[
\gamma_j=d\zeta_j-\frac{\bar\zeta\cdot d\zeta}{|\zeta|^2}\zeta_j, \quad j=1,\ldots,N,
\]
are projective forms such that 
\begin{equation}\label{olle2}
\nabla_\eta \gamma_j=2\pi i(z_j-\alpha \zeta_j),
\end{equation}
where $\alpha$ is the form in Example~\ref{alphaex}.

If $h(w,z)$ is a homogeneous form in $\C^{N+1}_w\times\C^{N+1}_z$ with differentials $dw$ 
and polynomial  coefficients, we let $\tau^*h$ be the projective 
form obtained by replacing $w$ by $\alpha\zeta$ and $dw_j$ by $\gamma_j$.
(Notice that thus $\tau^*$ does {\it not} necessarily preserve bidegree).
We then have 
\begin{equation}\label{tau}
\nabla_\eta \tau^*h=\tau^*(\delta_{w-z}h),
\end{equation}
in light of (\ref{olle2}) and \eqref{alfasluten}.

\begin{prop}[\cite{AG}, Proposition~4.4]\label{trams}
Assume that $\kappa\ge d_k^j$ for all $k$ and $j$. Then 
\begin{equation}\label{skola}
(H^E_\kappa)^\ell_k=\sum_{ij}(\tau^\ast h_k^\ell)_{ij}\w \alpha^{\kappa-d_k^j}
e_{\ell,i}\otimes e^*_{k,j}
\end{equation}
is a  Hefer morphism for the complex 
$E_\bullet\otimes\mathcal{O}(\kappa)$. 
\end{prop}

\begin{remark}\label{polly}
For degree reasons it is enough that $\kappa\ge d_k^j$ for 
$k\le\min(M,N+1)$. 
\end{remark}

It follows from the definition, cf., \cite{A7} or \cite{AW1},  that
the currents $U$ and $R$ that are associated to $E_\bullet$ are
also the currents associated to $E_{\bullet}\otimes\Ok(\kappa)$. 
From Propositions~\ref{trams} and \ref{burdus} we thus  get a division formula 
for sections $\psi$ of $E_0\otimes\Ok(\kappa-N)$.   

\section{Integral representation on  $X$}
\label{gam}

Let $i\colon X\to \Pk^N$ be a projective subvariety of pure dimension $n$. We shall now describe
how one can obtain intrinsic weighted representation formulas on $X$. With
an appropriate choice of weight we will finally get the desired  division formula in
Section~\ref{basic}.

Let $\E^{\Pk^N}_{\ell,q}$ denote  the sheaf of smooth $(\ell ,q)$-forms on  $\Pk^N$ and define
the sheaf $\E^X_{\ell,q}:=\E^{\Pk^N}_{\ell,q}/\Kers i^*$ of smooth forms on $X$,
where 
$\Kers i^*$ is the subsheaf of $\E^{\Pk^N}_{\ell,q}$ of forms whose pullback to $X_{reg}$ vanish. One can
show that $\E^X_{\ell,q}$ is an intrinsic sheaf over $X$,  that is, independent of the choice of embedding
in a smooth manifold.  We define the sheaf of currents of bidegree $\Cu^X_{n-\ell,n-q}$ on $X$ as the dual
of $\E^X_{\ell,q}$.  This means that the elements  $\tau\in \Cu^X_{n-\ell,n-q}$ precisely correspond
to the currents  $\hat\tau\in\Cu^{\Pk^N}_{N-\ell,N-q}$ such that $\hat\tau.\xi=0$ for $\xi\in\Kers i^*$.
It is natural to think of $\hat\tau$ as the push-forward $i_*\tau$ of $\tau$.

Let $\J_X\subset\Ok^{\Pk^N}$ be a coherent ideal sheaf associated with $X$.
From the free resolution \eqref{skottavla} of $S/J_X$ (of minimal length $M\le N$)
we can form  the 
locally free sheaf complex  $\Ok(E_\bullet)$ as in \eqref{utsikt},  defined by 
\begin{equation}\label{prost}
E_k=\big(E^1_k\otimes \mathcal{O}(-d^1_k)\big)\oplus\cdots \oplus
\big(E^{r_k}_k\otimes \mathcal{O}(-d_k^{r_k})\big),
\end{equation}
where $E_k^i$ are disjoint trivial line bundles, and the mappings $a_j$
are (formally) the same mappings as in \eqref{skottavla}.
It turns out, see, e.g., \cite[Section~6]{AW1},  that \eqref{prost} is actually
a resolution of the sheaf $\Ok^{\Pk^N}/\J_X$.
Let $R$ and $U$ be the associated currents as above and recall that
$R=R_p+R_{p+1}+\cdots $, where
$p:=N-n$.

Notice that   
$$
\varOmega := \delta_\zeta \big(d\zeta_0\w\cdots\w d\zeta_{N+1}\big)=
\sum (-1)^j\zeta_j d\zeta_0\w\ldots \w\widehat{d\zeta_j}\w\ldots\w d\zeta_N
$$
is a non-vanishing section of the trivial bundle over $\Pk^N$ realized as an
$(N,0)$-form with values in $\Ok(N+1)$.
It follows from \cite[Proposition~3.3]{AS} that  there is a unique 
current $\omega=\omega_0+\cdots+\omega_n$, where  $\omega_k$ has bidegree  
$(n,k)$ and takes values in $E_k\otimes\Ok(N+1)$, such that
\begin{equation}\label{deffo}
i_*\omega=\varOmega\w R, \quad  i_*\omega_\ell=\varOmega\w R_{p+\ell}.
\end{equation}
The current $\omega$  is called a {\it structure form} for $X$ in \cite{AS}.
It is smooth on $X_{reg}$, cf.,  \cite[Proposition~3.3]{AS} , and moreover we have:

\begin{lma}\label{snoklus}
Let $\chi(t)$ be a smooth function on $[0,\infty)$ that is $0$ close to $0$ and $1$ for large $t$.
If $\xi$ is a test form and  $h$ is a holomorphic
section of a Hermitian vector bundle that
does not vanish identically  on any irreducible component of
$X$,  then 
\begin{equation}\label{apfot}
\int_X\omega\w\xi=\lim_{\epsilon\to 0}\int_X \chi(|h|^2/\epsilon) \omega\w \xi.
\end{equation}
\end{lma}

Here the integrals denote the action of $\omega$ on test forms.
The lemma is one way to express that $\omega$ has the so-called {\it standard extension property}.

If we choose
$h$ that vanishes on $X_{sing}$, then $ \chi(|h|^2/\epsilon) \omega$ is smooth
for each $\delta>0$, and thus \eqref{apfot} defines
$\omega$ as a principal value current.

\begin{proof}
From \cite[Proposition~3.3]{AS} we know 
that $\omega$  is {\it almost semi-meromorphic} on $X$, cf.,  \cite[Definition~2.5]{AS}.
It follows from  \cite[Eq.\ (2.6)]{AS} that an alomst semi-meromorphic current 
has the standard extension property
according to the definition on p.\ 269 in \cite{AS}.  
From \cite[Eq.\ (2.3)]{AS}, and the discussion preceding
this formula the statement in the lemma now follows. 
\end{proof}

\begin{remark} 
Locally one can find a free resolution of $\Ok^{\Pk^N}/\J_X$  that ends at level
$N-1$. Then the top degree term of the associated residue current $R$ is $R_{N-1}$ and  therefore
the associated local structure form in \cite[Proposition~3.3]{AS} has top degree term 
$\omega_{n-1}$.  
When the residue current is constructed from a global resolution we usually get a term  $R_N$, 
and hence also a top term $\omega_n$ of our
global structure form. However, it
follows from the proof of \cite[Proposition~3.3]{AS} that $\omega_n$ carries no additional singularitites;
it is just a smooth form times $\omega_{n-1}$.
\end{remark}

For any smooth form $\xi$ on $\Pk^N$  there is a unique
form  $\sko(\xi)$ such that 
\begin{equation}\label{0skunk}
\sko(\xi)\w\varOmega=\xi_{N,*},
\end{equation}
where $\xi_{N,*}$ is the component of $\xi$ of bidegree $(N,*)$.
From \eqref{deffo} and  \eqref{0skunk}  we have that
\begin{equation}\label{skunk}
\xi_{N,*}\w R=\sko(\xi)\w\varOmega\w R=i_*\big(\sko(\xi)\w\omega\big),
\end{equation}
where we in the last term, for simplicity, write $\sko(\xi)\w\omega$ rather then
$i^*\sko(\xi)\w\omega$.

\smallskip

Let 
\begin{equation}\label{kappanoll}
\kappa_0:=\max_{i,k} d_k^i=\max_{k\le M, i} d_k^i.
\end{equation}
From Proposition~\ref{trams} we have a Hefer  morphism
$H^E_{\kappa_0}$ for the complex $E_\bullet\otimes\Ok(\kappa_0)$. 

\begin{prop}\label{primat}
Let $\ell$ be any integer and assume  that $g$ is a smooth weight on $\Pk^N$ with respect to
$\Ok(\ell-\kappa_0+N)$ and   $z\in X$. For 
holomorphic sections  $\phi$ 
of $\Ok(\ell)$ over $X$  we  have the representation
\begin{equation}\label{jupiter}
\phi(z)=\int_X \sko(g\w H^E_{\kappa_0})\w\omega \phi.
\end{equation}
\end{prop}

\begin{proof}
In view of Lemma~\ref{pluto} and Proposition~\ref{trams}, recalling that $a_j$ are the mappings in the complex defined by \eqref{prost}, 
$$
g^\lambda=a_1(z)H^E_{\kappa_0}U^\lambda+H^E_{\kappa_0} R^\lambda
$$ 
is a smooth weight in $\Pk^N$ with respect to $\Ok(\kappa_0)$ and $z$
if $\Re\lambda\gg 0$. Take $z\in  X$.  Then $a_1(z)=0$ and thus
\begin{equation}\label{kork}
g^\lambda=H^E_{\kappa_0} R^\lambda.
\end{equation}
Let $\Phi_0$ be a  smooth  global section of $\Ok(\ell)$ 
that is equal to $\phi$ on $X$ and holomorphic in a \nbh (in $\Pk^N$) of $z$. 
Since $B$ is smooth outside $z$, $\nabla_\eta B=1$ there and $\dbar\Phi_0=0$ in 
a \nbh (in $\Pk^N$) of $z$,  it follows that 
$$
\Phi:=\Phi_0-\dbar\Phi_0\w B
$$
is a smooth $\nabla_\eta$-closed section of $\Lop^0(\Ok_\zeta(\ell)\otimes\Ok_z(0))$
on  $\Pk^N$.  
Thus 
$$
\nabla_\eta(g^\lambda\w g\w \Phi \w B)=g^\lambda\w 
g\w \Phi(1-[z])=
g^\lambda\w g\w \Phi-(g^\lambda\w g\w \Phi )_{0,0}\w [z].
$$
Notice that $(g^\lambda\w g\w \Phi)_{0,0}$ is equal to $\phi(z)$ at the point $z$.
By a similar argument as in the discussion preceding
Section~\ref{krut} we get the representation
$$
\phi(z)=\int_{\Pk^N}g^\lambda\w g\w \Phi.
$$
This is essentially Proposition \ref{hatsuyuki}, except for that $\Phi$ is not holomorphic
but merely $\nabla_\eta$-closed.
Taking $\lambda=0$ we get
\begin{equation}\label{lollipop}
\phi(z)=\int_{\Pk^N}H^E_{\kappa_0} R\w g\w \Phi= 
\int_X\sko(g\w H^E_{\kappa_0})\w\omega \phi,
\end{equation}
cf., \eqref{skunk} and \eqref{kork},  since $i^*\Phi=\phi$.  
\end{proof}

\begin{ex}\label{tusentalet}
If $\ell\ge\kappa_0-N$, then we can take $g=\alpha^{\ell-\kappa_0+N}$
in Proposition~\ref{primat}, if $\alpha$ is the form in Example \ref{alphaex}. 
Since this weight is holomorphic for all  $z\in\Pk^N$,
the integral in \eqref{jupiter} then provides a holomorphic extension to 
$\Pk^N$ of $\phi$. We thus get an ``explicit'' proof of the surjectivity
of  the natural restriction  mapping
$$
\Gamma(\Pk^N,\Ok(\ell))\to \Gamma(X,\Ok(\ell))
$$
for  $\ell\ge \kappa_0-N$.
\end{ex}

For future reference, notice, cf., \eqref{badkar} and \eqref{kappanoll}, that  
\begin{equation}\label{kappanollett}
\kappa_0-N\le \reg X-1.  
\end{equation}

\section{Division formulas on $X$}\label{basic}

Let $f_1,\ldots,f_m$ be our 
sections of $\Ok(d)$ on $X$ from Section~\ref{intro}
and let $\J_f$ be the associated ideal sheaf on $X$.
If $X$ is smooth we can find
a resolution of $\Ok^X/\J_f$ over $X$ and obtain 
a residue current whose annihilator is precisely  $\J_f$.
We can then form a division-interpolation
formula like \eqref{divint},  if $\rho$ is big enough, 
such that the residue term vanishes as soon as $\phi$ is in $\J_f$. 
However, our objective is to construct an explicit representation
of the solutions in Theorem~\ref{hickelsats}, and to this end we use
the Koszul complex generated by the $f_j$.
In this way we get a residue current that is explicitly defined and, moreover,
perfectly adapted to that theorem.

\smallskip
Let $E'$ be a trivial rank $m$ bundle with basis elements $e_1,
\ldots, e_m$, let $E:=E'\otimes\Ok(-d)$, and let 
$e_i^*$ be the dual basis elements for $(E')^*$ so that
$f:=f_1e_1^*+\cdots+f_me_m^*$ is a section of $E^*=\Ok(d)\otimes (E')^*$.
We then have the  Koszul complex
\begin{multline}\label{staty}
0\to \Ok(-md)\otimes\Lambda^m E'\stackrel{\delta_f}{\to}\cdots \\
\cdots\stackrel{\delta_f}{\to}\Ok(-2d)\otimes\Lambda^2E'\stackrel{\delta_f}{\to}
\Ok(-d)\otimes E'\stackrel{\delta_f}{\to} \C\to 0,
\end{multline}
where $\delta_f$ denotes\footnote{These mappings $\delta_f$ are thus instances  of
the mappings  $\vp_j$ in  Sections~\ref{krut} and
\ref{hefersubsec}.} contraction by $f$.
If we let $E_k:=\Ok(-dk)\otimes \Lambda^kE'$ we thus have a complex
like \eqref{ecomplex} and the associated currents 
$U=U^f$ and $R=R^f$. Let us describe them in more detail.
In $\Pk^N\setminus \{f=0\}$ we 
define the section
$$
\sigma=\sum_{j=1}^m\frac{\overline{f_j(z)} e_j}{|f(z)|^2}
$$
of $E$.
If  $\bar f\cdot e:=\sum \bar f_j e_j$ and $d\bar f\cdot e:=\sum d\bar f_j\w e_j$, then
$$
\sigma\w(\dbar\sigma)^{k-1}=\frac{\bar f\cdot e\w (d\bar f\cdot e)^{k-1}}{|f|^{2k}}.
$$
It turns out, cf., the proof of Proposition~\ref{burdus} in Section~\ref{krut} 
and \cite[Example~1]{AG}, that 
\begin{equation}\label{oskar}
U^{f,\lambda}:=|f|^{2\lambda}_{E^*}
\sum_{k=1}^{m} \frac{\bar f\cdot e\w (d\bar f\cdot e)^{k-1}}{|f|^{2k}}
\end{equation}
and 
$$
R^{f,\lambda}:=1-|f|^{2\lambda}_{E^*}+\dbar |f|^{2\lambda}_{E^*}\w
\sum_{k=1}^{m} \frac{\bar f\cdot e\w (d\bar f\cdot e)^{k-1}}{|f|^{2k}}
$$
are  smooth currents for $\Re\lambda\gg 0$ and 
$$
R^f= R^{f,\lambda}|_{\lambda=0}, \quad U^f=U^{f,\lambda}|_{\lambda=0}.
$$

From  Section~\ref{hefersubsec}, cf., Remark~\ref{polly},  we know that
there is a Hefer morphism  $H^{f}_{\kappa}$ for the complex \eqref{staty}
if $\kappa\ge\kappa'$, where 
$$
\kappa':=d\min(m, N+1),
$$
and by Lemma~\ref{pluto}, 
\begin{equation}\label{snor}
g^{f,\lambda}_{\kappa}:=f(z)\cdot H^{f}_{\kappa} U^{f,\lambda}
+H^{f}_{\kappa}R^{f,\lambda}
\end{equation}
is then a smooth global weight on $\Pk^N$  with respect to $\Ok(\kappa)$ 
 when $\Re\lambda \gg 0$.

\begin{ex} 
Let us describe an explicit choice of  $H^{f}_{\kappa}$.
Let $ \tilde h_j(w,z)$ be $(1,0)$-forms in $\C^{n+1}\times\C^{n+1}$ 
of polynomial degrees $d-1$ such that
$$
\delta_{w-z} \tilde h_j=f_j(w)-f_j(z)
$$
and let 
$
h_j=\tau^*\tilde h_j$.  
Let
$$
h:=h_1\w e_1+\cdots + h_m\w e_m.
$$
If $\delta_h$ is interior multiplication by $h$ 
and 
$$
(\delta_h)_k:=(\delta_h)^k/k!,
$$
then we can take
\begin{equation}\label{luma}
(H^f_{\kappa})^\ell_k=\alpha^{\kappa-dk}(\delta_h)_{k-\ell},
\end{equation}
see \cite[Section~5]{AG}. 
\end{ex}

If now $\phi$ is a section of $\Ok(\rho)$ over $X$ 
and $\hat g$ is a weight on  $\Pk^N$ with  respect to
$\Ok(\rho-\kappa'-\kappa_0+N)$ and $z\in X$ it follows from Proposition~\ref{primat} that 
\begin{equation}\label{urskata0}
\phi(z)=\int_X \sko(g^{f,\lambda}_{\kappa'}\w \hat g\w H^E_{\kappa_0})\w\omega\phi.
\end{equation}
In particular, we can choose $\hat g$ that is holomorphic in $z$
if $\rho\ge \kappa'+\kappa_0-N$, cf., Example~\ref{tusentalet}.  

If $X$ is irreducible we can actually improve $\kappa'$ to 
$$
\kappa_1:=d\min(m,n+1).
$$
Notice that for fixed $z$ the scalar term in the form $\alpha$ is non-vanishing outside $\zeta=z$.  
Thus  $\alpha^{-1}$ is  well-defined there, and so we can define
$g^{f,\lambda}_{\kappa}$ by \eqref{snor} and  \eqref{luma} for any $\kappa$ there. 

\begin{lma}\label{skatan}
The pullback to $X$ of  $g^{f,\lambda}_{\kappa_1}$ is smooth on $X$.
Assume that $X$ is irreducible.
If  $g$ is a weight with respect to  $\Ok(\rho-\kappa_1-\kappa_0+N)$
and $\phi$ is a holomorphic section of $\Ok(\rho)$ over $X$, then
we have the representation
\begin{equation}\label{styrke}
\phi(z)=\int_X \sko(g^{f,\lambda}_{\kappa_1}\w g\w H^E_{\kappa_0})\w\omega\phi, \quad z\in X.
\end{equation}
\end{lma}

\begin{proof}
Notice that terms of $g^{f,\lambda}_{\kappa}$ of
bidegree $(s,s)$,   $s>n$,  vanish on $X$.
Recall that in the definition \eqref{snor}, 
 $(H^f_\kappa)^\ell_\bullet$ occurs with  $\ell=1$ in the first term and with 
 $\ell=0$  in the second one. Thus the ``worst'' component
of \eqref{luma} that occurs in   $g^{f,\lambda}_{\kappa_1}$ and does not vanish
on $X$, is when $\ell=1$ and $k=n+1$.

Fix  $z\in X$.  Notice that the section  
$h(\zeta)=\zeta\cdot\bar z/|z|$ of $\Ok(1)$ is non-vanishing at $\zeta=z$.
Let $\chi(t)$ be a cutoff function as in Lemma~\ref{snoklus} and let 
$$
\chi_\delta:=\chi(|h|^2/\delta)=\chi\big(|\bar\zeta\cdot z|^2/|z|^2|\zeta|^2\delta\big).
$$
Here $|h|=|h|_{\Ok(1)}$ denotes the natural norm as a the section of $\Ok(1)$, cf., \eqref{norm}. 
For small $\delta$, $\dbar\chi_\delta$ vanishes in a \nbh of $z$, and thus
\begin{equation}\label{lilian}
g^\delta:=\chi_\delta-\dbar\chi_\delta\w B
\end{equation}
is a smooth weight with respect to $z$, cf., \eqref{pelle}. 
Since  $g^\delta$ vanishes in a \nbh of the hyperplane
$h=0$  it follows that
$$
\alpha^{-r}\w g^\delta
$$
is a smooth weight with respect to $\Ok(-r)$ and $z$ for any $r$. Notice that
$$
g^{f,\lambda}_{\kappa'}=\alpha^{\kappa'-\kappa_1}\w g^{f,\lambda}_{\kappa_1}.
$$
Thus
we can choose  $\hat g=\alpha^{-r}\w g^\delta \w g$
in \eqref{urskata0} with $r=\kappa'-\kappa_1$ and so we get 
\begin{equation}\label{urskog}
\phi(z)=\int_X \sko(g^{f,\lambda}_{\kappa_1}\w g^\delta\w g\w H^E_{\kappa_0})\w\omega\phi.
\end{equation}

Since $h\neq 0$ at $z$ and $X$ is assumed to be irreducible
$h$ is non-vanishing generically on $X$.   
From Lemma~\ref{snoklus} it follows that 
\begin{equation}\label{ratta1}
\chi_\delta \omega\to \omega
\end{equation}
when $\delta\to 0$. We also know from \cite[Proposition~3.3]{AS} 
(recall that $a$ denotes the mappings in Section~\ref{gam} that define
the resolution of $\Pk^N/\J_X$) 
then $\nabla_a \omega=0$. Therefore
\begin{equation}\label{ratta2}
-\dbar\chi_\delta\w\omega=\nabla_a (\chi_\delta\w\omega)\to
\nabla_a\omega=0.
\end{equation}
In view of \eqref{ratta1}, \eqref{ratta2} and \eqref{lilian}, keeping in mind that
$g^{f,\lambda}_{\kappa_1}$ is smooth,  we  get \eqref{styrke} from
\eqref{urskog} when $\delta\to 0$.
\end{proof}

Notice that $\sko$ only acts on factors with
holomorphic  differentials so that $R^{f,\lambda}$ (and $U^{f,\lambda}$) 
can be put outside the brackets in \eqref{styrke}.

One can define the (formal)  product currents 
$$
U^{f}\w\omega:=U^{f,\lambda}\w\omega\big|_{\lambda=0},
\quad  
R^{f}\w\omega:=R^{f,\lambda}\w\omega|_{\lambda=0},
$$
on $X$, see  \cite[Sections~2.7 and 2.5]{AWsemester}.
The first one is a product of two almost semi-meromorphic currents,
cf., \cite[Definition 2.5]{AS}, 
and it is robust in the sense that 
\begin{equation}\label{kabel}
U^{f}\w\omega=\lim_\delta \chi_\delta U^f\w \omega
\end{equation}
if $\chi$ is a function as in Lemma~\ref{snoklus},
$\chi_\delta=\chi(|h|^2/\delta)$  and $h$ is a holomorphic section
that is generically non-vanishing on $X$.  This follows from
\cite[Proposition~2.7]{AS} combined with \cite[Eqs.\ (2.2) and (2,3)]{AS}.  
In particular we can choose a section $h$ that vanishes on $Z_f\cup X_{sing}$;
recall that  $Z_f$ is the zeroset of $f$ on  $X$.  Then
$\chi_\delta U^f\w \omega$ is smooth for $\delta>0$ and thus
$U^f\w\omega$ is a principal value current.

Putting $\lambda=0$ in \eqref{styrke} we  get the following 
more precise form of \eqref{divint} in the introduction.

\begin{thm}\label{puma}
Assume that $X$ is irreducible. Let $f_1,\ldots, f_m$ be sections of $\Ok(d)$.  
If 
$$
\rho= d\min(n,m+1)+\kappa_0-N +\ell=\kappa_1+\kappa_0-N+\ell, \quad  \ell\ge 0,
$$
and  $\phi\in\Gamma(X,\Ok(\rho))$
we have the interpolation-division formula
\begin{equation}\label{divint2}
\phi(z)=f(z)\cdot \int_X \Ak^\rho(\zeta,z)\phi(\zeta)+
\int_X \Bk^\rho(\zeta,z)\w  R^f(\zeta)\w\omega(\zeta)\phi(\zeta),
\end{equation}
where
$$
\Ak^\rho(\zeta,z):=\sko\big(\alpha^\ell\w H^f_{\kappa_1}\w H^E_{\kappa_0}\big)
\w U^f\w\omega
$$
and
$$
\Bk^\rho(\zeta,z):=\sko\big(\alpha^\ell\w H^f_{\kappa_1}\w H^E_{\kappa_0}\big).
$$
\end{thm}

As before $\alpha$ is the form in Example \ref{alphaex}. 
From above it is clear that $\Ak^\rho(\zeta,z)$ is an $m$-tuple that is 
holomorphic in $z\in X$ (with values in $\Ok(\rho-d)$),
and that, for fixed $z$, it is the product
of the principal value current $U^f\w\omega$ and a smooth form.  
Moreover, $\Bk^\rho(\zeta,z)$
is smooth and holomorphic in $z\in X$ (with values in $\Ok(\rho)$).
If $R^f\w\omega \phi=0$, then the second integral in \eqref{divint2} vanishes and so
we get a solution $q$ to   $f\cdot q=\phi$, where the $m$-tuple $q$ is given by the principal value
integral
$$
q(z)= \int_X \Ak^\rho(\zeta,z)\phi(\zeta).
$$
Recall that $(H^f_\kappa)^\ell_\bullet$ occurs
in \eqref{divint2} with $\ell=1$ in the first integral and with $\ell=0$  in the
second one. In  view of \eqref{oskar} and the special choice of Hefer
morphism \eqref{luma} 
we have in particular that 
\begin{equation}\label{hurra3}
q=\sum_{k=1}^{\min(m,n+1)}
\int_X\sko\Big[\alpha^{\rho-\kappa_0+N-dk}\w (\delta_{h})_{k-1} 
\frac{\bar f\cdot e\w(d\bar f\cdot e)^{k-1}}
{|f|^{2k}}\w H^E_{\kappa_0}\Big]\w \omega\phi,
\end{equation}
where the integral is a principal value at  $Z_f\cup X_{sing}$. 


\begin{remark}
If $m\le n$ and  $|\phi|\le |f|^{\min(m,n)}$, then 
$U^f\phi$ is  integrable  on $X_{reg}$ so 
\eqref{hurra3}  is  a convergent integral locally on $X_{reg}$.
In general,
however, $U^f\phi$ may be  a distribution of  higher order than zero,
and then the integral in \eqref{hurra3} must be regarded as a principal value
even at $Z_f$ on $X_{reg}$.
\end{remark}

\begin{remark}
If $X$ is smooth, then $\omega$ is smooth as well.
If  in addition $\codim Z_f=m$, then 
$R^f\w\omega\phi=0$ if and only if
$\phi$ is in the sheaf $\J_f$, see, e.g.,  \cite{AW1}. 
Moreover, $R^f=R^f_m$ 
coincides with the classical Coleff-Herrera product
$$
\dbar\frac{1}{f_m}\w\ldots\w\dbar\frac{1}{f_1}
$$
on $X$, cf., \cite{A3} p.\ 112. 
\end{remark}

\section{Proof of Theorem~\ref{hickelsats}}\label{ponton}
The key step in the proof of  Theorem~A
in \cite{AWsemester} is the following result, which follows from 
the proof of Theorem~A in  \cite[Section~6]{AWsemester}:

\begin{prop}\label{orangutang} 
Let $V$ be a reduced $n$-dimensional algebraic subvariety of $\C^N$ and let
$X$ be its closure in $\P^N$. 
There is a number $\mu_0$, only depending on $V$, such that the following
holds: If $F_j$ are polynomials of degree at most $d$ and $\Phi$ is a polynomial such that
\eqref{kaka} holds, $f_j$ are the $d$-homogenizations of $F_j$,
$\rho\ge  \deg \Phi + (\mu+\mu_0) d^{c_\infty} \deg X$, 
and $\phi$ is the $\rho$-homogenization
of $\Phi$, then $R^f\w\omega\phi=0$.

Assume that $V$ is smooth. There is a number  $\mu'$ such that if 
$F_1,\ldots,F_m$ are polynomials of degree
$\le d$, $f_j$ are the $d$-homogenizations of $F_j$,
$\Phi$ is a polynomial such that \eqref{kaka2} holds, 
$\rho\ge \deg \Phi + \mu  d^{c_\infty}\deg X + \mu'$, 
and $\phi$ is the $\rho$-homogenization
of $\Phi$, then $R^f\w\omega\phi=0$.
If $X$ is smooth one can take $\mu'=0$. 
\end{prop}

Now Theorem~\ref{hickelsats} follows with the same constants $\mu_0$ and $\mu'$. 
In fact, consider the formula \eqref{divint2} with  $\rho$ as in \eqref{rodef}, 
which is an allowed choice in view of \eqref{kappanollett}. 
Assume that $F_j$ and $\Phi$ are such that
\eqref{kaka} holds and let $\phi$ be the $\rho$-homogenization of $\Phi$.
By Proposition~\ref{orangutang}, 
 $R^f\w\omega\phi=0$, and hence
$$
\phi=f(z)\cdot\int \Ak^\rho(\zeta,z).
$$
By dehomogenization we get the desired representation of the membership.
Part (ii) of  Theorem~\ref{hickelsats} follows in the same way by taking
$\rho'$ as in \eqref{rodef2}.
It is clear from these arguments that one can replace the number $\reg X-1$ by 
$\kappa_0-N$ in Theorem~\ref{hickelsats}.

\section{The case when $X$ is a (reduced) hypersurface}
\label{hypers}

In this section we illustrate the results in the special case when $X$ is
a reduced hypersurface  in $\Pk^{n+1}$.
We thus assume that $X=\{a=0\}$ where $a$ is a section of $\Ok(\kappa_0)$
in $\Pk^{n+1}$,
i.e., $a=a(\zeta_0,\ldots,\zeta_{n+1})$ is a $\kappa_0$-homogeneous
polynomial
in $\C^{n+2}$, and $d a\neq 0$ on (the pull-back to $\C^{n+2}$ of) $X$.
We first discuss  the general representation formula
in Proposition~\ref{primat}  in this case.  It can of course be
obtained from  this proposition  but we find it instructive
to derive it directly from
the general representation formula (Proposition~\ref{hatsuyuki}) on $\Pk^{n+1}$. 

Recall that $\|a\|:=|a|/|\zeta|^{\kappa_0}$ is the natural pointwise norm of $a$ considered
as a section of $\Ok(\kappa_0)$. It is well-known that
$\dbar\|a\|^{2\lambda}/a$, a ~priori defined for $\Re\lambda \gg 0$, admits a current-valued
analytic continuation to $\Re\lambda>-\epsilon$, and that the value at $\lambda=0$
is $\dbar(1/a)$, i.e., $\dbar$ applied to the principal value current $1/a$.

\begin{remark}\label{plast}
We have the locally free sheaf resolution
\begin{equation} 0\rightarrow\mathcal{O}(-\kappa_0)\stackrel{a}{\rightarrow}\Ok(0)
\end{equation}
of  $\mathcal{O}^{\mathbb{P}^N}/\mathcal{J}_X$, and
it is readily checked that the associated residue current $R^E$ is
the current $\dbar(1/a)$ 
times a homomorphism (from $\Ok(0)$ to $\Ok(-\kappa_0)$) that has odd order. 
However we will derive the
representation  formula on $X$ in a more direct way 
and then we do not have to bother about these sign problems. 
Notice that the number $\kappa_0$ here is consistent with the definition in the general case.
Also notice that $\reg X= d_1^1-1 +1=\kappa_0$, cf., \eqref{badkar}.  
\end{remark}

Let $h^a(w,z)$ be a $(1,0)$-form 
such that $\delta_{w-z} h^a=a(z)-a(w)$ on $\C^{n+2}\times\C^{n+2}$, and 
let us write $h^a(\alpha\zeta,z)$ for $\tau^*h$, cf, Section~\ref{hefersubsec}.
If $\Re\lambda \gg 0$, then
$$
g^\lambda:=\alpha^{\kappa_0}-\nabla_\eta\big(h^a(\alpha\zeta,z)\|a\|^{2\lambda}/a\big)
$$
is a weight in $\Pk^{n+1}$ with
respect to $\Ok(\kappa_0)$ and  $z$. If $\phi$ is a holomorphic section
of $\Ok(\ell)$ and $g$ is a weight with respect to
$\Ok(\ell-\kappa_0+n+1)$, then 
we have from Proposition~\ref{hatsuyuki}, with $F=\Ok(n+1+\ell)$, the representation
$$
\phi(z)=\int_{\Pk^{n+1}} g^\lambda\w g \phi.
$$
Since
$$
-\nabla_\eta h^a(\alpha\zeta,z)=a(z)-a(\alpha\zeta)=a(z)-\alpha^{\kappa_0}a(\zeta),
$$
we have that
$$
g^\lambda=(1-\|a\|^{2\lambda})\alpha^{\kappa_0}+\frac{a(z)}{a(\zeta)}\|a\|^{2\lambda}+
h^a(\alpha\zeta,z)\w \dbar\|a\|^{2\lambda}/a.
$$
Notice that
$$
\int (1-\|a\|^{2\lambda})\alpha^{\kappa_0}\w g\phi
$$
vanishes when $\lambda=0$ by the dominated convergence theorem.
Let us now assume that $z\in X$  so that $a(z)=0$.
We then have 
\begin{equation}\label{ug}
\phi(z)=\int_{\Pk^{n+1}} g\w h^a(\alpha\zeta,z)\w\dbar\frac{1}{a}\phi, \quad z\in X.
\end{equation}
Arguing as in the proof of Proposition~\ref{primat} 
it is enough to assume that $\phi$ a~priori is defined on $X$.

We  want to write the right hand side  in \eqref{ug} as a principal
value integral over   $X$.   As before, let 
$$
\varOmega:=\delta_\zeta d \zeta,
$$
where
$$
d\zeta:=d\zeta_0\wedge\ldots\wedge d\zeta_{n+1}.
$$
Recall, cf., \eqref{deffo},
that the form $\omega$ on $X$ is defined by the equality\footnote{Since $R^E$ has even degree
this is consistent  with \eqref{deffo},
cf. Remark~\ref{plast}.}
\begin{equation}\label{alla}
i_* \omega =\dbar(1/a)\wedge\varOmega.
\end{equation}
We shall give  an explicit representation of $\omega$.
Let
$$
\partial_ja=\frac{\partial a}{\partial \zeta_j},
\ j=0,\ldots,n+1, \quad |\partial a|^2=|\partial_0a|^2+\cdots+|\partial_{n+1}a|^2,
$$
let $\delta_{A}$ denote interior multiplication by
$$
2\pi i\frac{1}{|\partial a|^2}\sum_{j=0}^{n+1} \overline{\partial_ja}\partial_j,
$$
and define
\begin{equation}\label{glass}
\omega':=\delta_{A}\varOmega=\delta_{A}\delta_\zeta d\zeta.
\end{equation}

\begin{lma} \label{strut}
For any test form  $\xi$ we have that
\begin{equation}\label{postbud}
\int_{\Pk^{n+1}}\xi\w \dbar\frac{1}{a}\w\varOmega=\int_X
\xi\w\omega'. 
\end{equation}
\end{lma}

In view of \eqref{alla} we thus have that
$$
\omega=i^*\omega'.
$$

\begin{proof}
Notice that 
$$
Da:=da-\kappa_0\frac{\bar \zeta\cdot d \zeta}{|\zeta|^2}a
$$
is the Chern connection on $\Ok(\kappa_0)$ acting on  $a$. One can verify directly
that $Da$ is a projective form, since by the $\kappa_0$-homogeneity of $a$,
$\partial_0a\zeta_0+\cdots+\partial_{n+1}a\zeta_{n+1}=\kappa_0a$.
In any case,  $\delta_\zeta(Da)=0$ so we have  
\begin{multline*}
Da\w\omega'=Da\w \delta_{A}\delta_\zeta d\zeta =\delta_\zeta(Da\w \delta_{A} d\zeta)=
\delta_\zeta(\delta_A Da\w d\zeta)=\\
2\pi i\Big(1-\kappa_0\frac{\overline{\partial a}\cdot\bar\zeta}{|\partial a|^2|\zeta|^2}a\Big) 
\varOmega.
\end{multline*}
In particular,
$$
Da\w \omega'=2\pi i\varOmega
$$
on $X$.  By the Poincar\'e-Lelong lemma,
$$
\dbar\frac{1}{a}\w Da=2\pi i[X],
$$
and therefore
$$
2\pi i\dbar\frac{1}{a}\w\varOmega=\dbar\frac{1}{a}\w Da\w\omega'=
2\pi i[X]\w\omega'
$$
which is the same as \eqref{postbud}.
\end{proof}

If $\xi$ has bidegree $(n+1,n)$, then by \eqref{glass}, 
$$
\delta_A\xi=\delta_A(\vartheta(\xi)\w\varOmega)=(-1)^n\vartheta(\xi)\w\omega'.
$$
Moreover, 
$$
\vartheta(\xi)\w\varOmega\w\dbar\frac{1}{a}=(-1)^{n+1}\vartheta(\xi)\w\dbar\frac{1}{a}\w\varOmega.
$$
From \eqref{ug}  and Lemma~\ref{strut} we  thus get

\begin{prop}
If $\phi$ is a holomorphic section of $\Ok(\ell)$ 
and $g$ is a weight with respect to $\Ok(\ell-\kappa_0+n+1)$, 
then we have the representation
\begin{equation}\label{tulpan}
\phi(z)=(-1)^{n+1} \int_X\sko(g\w h^a(\alpha\zeta,z))\w\omega' \phi=
-\int_X\delta_A(h^a\w g)\phi.
\end{equation}
\end{prop}

Given the  situation  in  Section~\ref{ponton}, 
the dehomogenization of
\begin{multline}\label{hurra4}
q=\sum_{k=1}^{\min(m,n+1)}
(-1)^{n+1}\\
\int_{a=0} \sko\Big[\alpha^{\rho-\kappa_0+n+1-dk}\w (\delta_{h})_{k-1} \frac{\bar f\cdot e\w(d\bar f\cdot e)^{k-1}}
{|f|^{2k}}\w h^a(\alpha\zeta,z)\Big]\w \omega'\phi.
\end{multline}
is thus a tuple of polynomials $Q_j$ such that \eqref{hummer} and
\eqref{polupp} hold on $X$.

\def\listing#1#2#3{{\sc #1}:\ {\it #2},\ #3.}

\end{document}